\documentclass{article}

\usepackage{enumerate}

\usepackage[centertags]{amsmath}
\usepackage{amsfonts}
\usepackage{amssymb}
\usepackage{amsthm}
\usepackage{newlfont}
\usepackage{mathtools}
\usepackage[top=1in, bottom=1in, left=1in, right=1in]{geometry}
\usepackage{tikz}

\theoremstyle{plain}
\newtheorem{theorem}{Theorem}[section]

\newtheorem{corollary}[theorem]{Corollary}
\newtheorem{lemma}[theorem]{Lemma}

\newtheorem{remark}[theorem]{Remark}

\theoremstyle{definition}
\newtheorem{definition}[theorem]{Definition}

\newcommand{\R}{\mathbb{R}}
\newcommand{\N}{\mathbb{N}}

\newcommand{\ep}{\varepsilon}

\newcommand{\Deg}{\operatorname{Deg}}
\newcommand{\D}{\Deg_{\max}}
\newcommand{\Qmin}{Q_{\min}}
\newcommand{\eps}{\varepsilon}

\newcommand{\Hidden}[1]{}

\DeclareMathOperator{\sgn}{sgn}
\DeclareMathOperator{\dfct}{def}

\DeclarePairedDelimiter{\abs}{\lvert}{\rvert}
\DeclarePairedDelimiter{\norm}{\|}{\|}

\begin{document}

\title{Large scale Ricci curvature on graphs}
\author{Mark Kempton\footnote{Brigham Young University, Provo UT, USA. mkempton@mathematics.byu.edu}
~~~~~Gabor Lippner\footnote{Northeastern University, Boston MA, USA. g.lippner@northeastern.edu}
~~~~~Florentin M\"unch\footnote{Max Planck Institute, Leipzig, Germany. Florentin.Muench@mis.mpg.de}}
\date{}
\maketitle

\begin{abstract}
We define a hybrid between Ollvier and Bakry Emery curvature on graphs with dependence on a variable neighborhood. The hexagonal lattice is non-negatively curved under this new curvature notion. Bonnet-Myers diameter bounds and Lichnerowicz eigenvalue estimates follow from the standard arguments. We prove gradient estimates similar to the ones obtained from Bakry Emery curvature allowing us to prove Harnack and Buser inequalities.
\end{abstract}

\section{Introduction}

The analogy between graphs and Riemannian manifolds, via their Laplace operators, has a long history~
\cite{mohar1989isoperimetric, alon1985lambda1, keller2010unbounded,mohar1989survey,
chung1997spectral,chung2005laplacians,
hein2005graphs,weber2008analysis}. Originally, this analogy was mostly exploited on the level of spectrum and eigenfunctions.  Recently, however, various notions of curvature for discrete spaces, and graphs in particular, have received growing attention~\cite{lin2011ricci,ollivier2007ricci,
ollivier2009ricci,erbar2012ricci,
najman2017modern,jost2014ollivier,
bauer2011ollivier,bauer2015li,
bhattacharya2015exact,lin2015equivalent,
liu2014eigenvalue,munch2014li,munch2017remarks,
paeng2012volume,chung2017curvature,
bourne2017ollivier,liu2016bakry,
munch2017ollivier}. Unfortunately, it seems, there is no single best notion of curvature in the discrete setting. Instead, one typically chooses a standard consequence of curvature bounded from below in the manifold setting, and uses that as the definition in the discrete setting. So far the most popular methods were based on optimal transport (Ollivier-Ricci curvature) or on Bochner's identity and curvature-dimension inequalities (Bakry-Emery curvature). While Bakry-Emery curvature conditions imply many interesting properties, the examples even where the simplest $CD(0,\infty)$ condition is known to hold are very limited. On the other hand, Ollivier curvature conditions, while allowing for more examples, are not known to imply many of the usual and important consequences that hold on manifolds.

The quest is then to find conditions that are strong enough to imply interesting theorems, but not too restrictive to allow many interesting examples. Our view is that the extreme and rigid locality (that curvature depends only on the 2-step neighborhood of a node) is one of the major roadblocks in finding interesting examples. For instance, all previously studied curvature variants will assign negative curvature to the hexagonal lattice, since it is locally tree-like up to two steps. However, it ``should'' clearly have non-negative curvature as it is a planar tiling.

In this paper we introduce three related, but slightly different, new notions of curvature. These are hybrids between Ollivier and Bakry-Emery curvatures. Their common, and perhaps most important, feature is that they allow a scaling parameter: the radius $R$ of the neighborhood of the node they are sensitive to. In particular, the hexagonal lattice will have non-negative curvature for $R = 2$. At the same time, lower bounds on these curvature will imply many of the consequences of Bakry-Emery curvature.

We derive gradient estimates for all variants, and give a combinatorial characterization of non-negative curvature. As it turns out, having non-negative curvature for any one variant implies non-negative curvature for the other two variants. So in the sense of non-negative curvature, they are equivalent to each other.

The paper is organized as follows. In Section~\ref{sec:first_estimate} we motivate and introduce our simplest new variant. In Section~\ref{sec:variants} we define the other two curvature notions, and in Section~\ref{sec:estimates} we explain important consequences of curvature bounded from below in all three cases. In Section~\ref{sec:nonnegative} we give a combinatorial characterization of non-negative curvature. Finally, Section~\ref{sec:proofs} contains all the proofs.

\subsection{Notation}\label{sec:notation}

Throughout the paper we consider simple graphs $G(V,E)$ with vertex set $V$ and edge set $E$, or weighted graphs $G(V,w,m)$ with vertex set $V$, (symmetric) edge weights $w : V\times V \to [0,\infty)$, and vertex measure $m : V \to (0,\infty)$. In the case of weighted graphs, the edge set is simply the set of pairs whose weight is positive. Simple graphs are recovered from the weighted version by setting $w(x,y) = 1$ for all edges $(xy) \in E$ and setting $m \equiv 1$.  

We will write most proofs for the case of simple graphs, but everything works identically for weighted graphs. We will point out when extra conditions on the weights and vertex measure are necessary.

We use $d(x,y)$ to denote the combinatorial distance of nodes $x,y \in V$. 
The Laplace operator is given by 
\[
\Delta f(x) = \frac 1 {m(x)} \sum_y w(x,y) (f(y)-f(x)),
\]
and the corresponding Bakry-Emery $\Gamma$ operator is defined as $\Gamma(f,g) = \frac{1}{2}(\Delta (fg) - f \Delta g - g \Delta f)$, so 
\[ 
\Gamma(f,g)(x) = \frac{1}{2 m(x)} \sum_y w(x,y) (f(y)-f(x))(g(y)-g(x)).
\]
As usual, $\Gamma f$ will stand for $\Gamma(f,f)$.

The (heat) semigroup generated by $\Delta$ will be denoted by $P_t$. For bounded $f:  V\to \R$ it satisfies the heat equation
\[ \partial_t P_t f = \Delta P_t f.\]
For further details about the heat semigroup, see e.g. \cite{keller2010unbounded, wojciechowski2008heat}. 


\subsection{The hybrid curvature}\label{sec:first_estimate}

In this section we explain the basic setup of our new curvature. The definition will be motivated by the following result~\cite[Theorem~2]{renesse2005transport} stating that
for smooth connected Riemannian manifolds $M$, a lower Ricci curvature bound $K$ is equivalent to
\[
|\nabla P_t f| \leq e^{-Kt} P_t |\nabla f|
\]
for all $f \in C_c^{\infty}(M)$, where $P_t$ is the heat semigroup. 

We aim for adapting this characterization to graphs. To do so, we have to introduce an appropriate gradient notion. This is the key definition that will allow the curvature to detect large (but fixed) neighborhoods of nodes.

\begin{definition}[$R$-Gradient]
Let $R \in \N$, and set $B_R(x) = \{ y : d(x,y)\leq R\}$. For any $f: V\to \R$ and $x \in V$ let us define \[
|\nabla_R f|(x) := \max_{d(x,y)\leq R} |f(y)-f(x)|.
\]
\end{definition}
We now derive necessary features of the geometry in order to have, for all bounded $f : V \to \R$, 
\begin{equation}\label{eq:linear_decay}
|\nabla_R P_t f| \leq e^{-Kt} P_t |\nabla_R f|.
\end{equation}
At $t=0$ the two sides of~\eqref{eq:linear_decay} are, of course, equal. Thus the derivative of the right-hand side has to be at least the derivative of the left-hand side.  Some care must be taken when computing the derivative, as the $\abs{ \nabla_R \, \cdot }$ is a maximum. However, at least the right derivative of both sides will exist and be computable. These derivatives are what we can then compare. \textbf{In the whole paper, derivative will always mean right derivative, unless otherwise noted.}  Fix $x \in V$. Since $B_R(x)$ is a finite set, there is an $\ep > 0$ such that $\abs{\nabla_R P_t f}(x)$ is attained at the same $y_0 \in B_R(x)$ for all $t \in [0,\ep]$ and such that $\sgn (P_t f(y) - P_t f (x)$ is constant on $(0,\ep]$ for any $y \in B_R(x)$. 

Let us fix such an $\ep > 0$ and let $y_0 \in N_R(x):=B_R(x)\setminus \{x\}$ such that \[ \forall t \in [0,\ep], \forall y \in N_R(x) :  \abs{\nabla_R P_t f}(x) = \abs{P_t f(y_0)-P_t f(x)} \geq  \abs{P_t f(y) - P_t f(x)}\] Then, at least at $t=0$, the derivatives can also be compared: $ \partial_t |P_t f(y_0) - P_t f(x)| \geq \partial_t \abs{P_t f(y) - P_t f(x)}$ for any $y \in N_R(x)$ such that $\abs{\nabla_R f}(x) = \abs{f(y)-f(x)}$. Further it is clear that  \[ \partial_t \abs{P_t f(y) - P_t f(x)} = \partial_t ((P_t f(y) - P_t f(x))\sgn(P_t f(y) - P_t f(x))) = (\Delta P_t f(y) - \Delta P_t f(x)) \sgn(f(y)-f(x)).\] Then for any $y \in N_R(x)$ such that $\abs{\nabla_R f}(x) = \abs{f(y)-f(x)}$, at $t=0$ we have
\begin{multline}\label{eq:partial_gradient} \partial_t \abs{\nabla_R P_t f}(x) = \partial_t |P_t f(y_0) - P_t f(x)| \geq \partial_t \abs{P_t f(y) - P_t f(x)}=\\ = \partial_t \left((P_t f(y) - P_t f(x))\sgn(f(y)-f(x))\right) = (\Delta P_t f(y) - \Delta P_t f(x))\sgn(f(y)-f(x)) \end{multline}
Hence, at $t=0$, by the comparison of the derivatives in~\eqref{eq:linear_decay} we get
\[
(\Delta f(y) - \Delta f(x))\sgn(f(y)-f(x)) \leq \partial_t |\nabla_R P_t f|(x) \leq \partial_t\left(e^{-Kt} P_t |\nabla_R f| \right)(x) = -K |\nabla_R f|(x) + \Delta |\nabla_R f|(x).
\]
So by assuming~\eqref{eq:linear_decay} we get that for all $f$ such that $|f(y)-f(x)|=|\nabla_R f|(x)$, the following inequality holds.
\begin{equation}\label{eq:linear_local}
\Delta  |\nabla_R f|(x) -  (\Delta f(y) - \Delta f(x))\sgn(f(y)- f(x)) \geq K  |\nabla_R f|(x)
\end{equation}

Here everything is homogeneous of degree 1 in $f$, so it suffices to consider functions where $\abs{\nabla_R f}(x) = 1$. The above calculations motivate the following definition of curvature.
\begin{definition}[Gradient-Ollivier curvature]\label{def:linear_local} For any $R \in \N$ and $x,y \in V$ we set 
\[
K_R(x,y) := \inf_{\substack{f : V \to \R \\ |\nabla_R f|(x)=1\\ |f(y)-f(x)| =  |\nabla_R f|(x)}}\Delta  |\nabla_R f|(x) -  (\Delta f(y) - \Delta f(x))\sgn(f(y)- f(x)).
\]
\end{definition}
Thus, if \eqref{eq:linear_decay} holds for all $f \in l^\infty(V)$, we see that $K_R(x,y) \geq K$ for all $x,y$ such that $d(x,y) \leq R$.  As we will see in Theorem~\ref{thm:linear_equivalent}, this is also a sufficient condition.

We will also show  that a lower bound on $K_R$ for $R=1$ is ``stronger'' than a lower bound on the classical Ollivier curvature $K^{Ol}$:

\begin{theorem}\label{thm:Ollivier_comparison}
For $x\sim y$, we have
\[
K_1^{Ol}(x,y) \geq K_1(x,y)
\]
\end{theorem}

\subsection{Further curvature variants}\label{sec:variants}
Considering different gradient estimates, one can find different notions of curvature. In this paper we consider two such variants.

\subsubsection{Quadratic curvature}
The following is a quadratic version of the gradient estimate~\eqref{eq:linear_decay}:
\begin{equation} \label{eq:quadratic_decay}
\abs*{\nabla_R \sqrt{ P_t f}}^2 \leq e^{-2K t}  P_t \abs*{ \nabla_R \sqrt{f} }^2
\end{equation}

Let $y \in N_R(x)$ such that $\abs{\sqrt{f}(y) - \sqrt{f}(x)} = \abs{\nabla_R \sqrt{f}}(x)$. 
A calculation similar in spirit to the one in Section~\ref{sec:first_estimate} yields that the following inequality is a consequence of~\eqref{eq:quadratic_decay}.
\[ 2K \abs{\nabla_R \sqrt{f}}^2 \leq \Delta \abs{\nabla_R \sqrt{f}}^2  - (\sqrt{f}(y)-\sqrt{f}(x))\left(\frac{\Delta f}{\sqrt{f}}(y)- \frac{\Delta f}{\sqrt{f}}(x)\right) .\]
Denoting $g = \sqrt{f}$ for convenience, and normalizing so that $\abs{\nabla_R g}(x)=1$, we get the following definition.

\begin{definition}[Quadratic Gradient-Ollivier curvature]\label{def:quadratic_local}
Fix an integer $R > 0$. 
Let $d(x,y)  \leq R$. The \emph{quadratic Gradient-Ollivier curvature} on the pair $xy$ is given by
\[ K_R^q(x,y) = \frac{1}{2} \inf_{\substack{g >0 \\ \abs{\nabla_R g}(x)= 1 \\ \abs{g(y)-g(x)}=1}} \Delta \abs*{\nabla_R g}^2(x) - \sgn(g(y)-g(x))\left(\frac{\Delta g^2}{g}(y) - \frac{\Delta g^2}{g}(x)\right)\]
\end{definition}

The above argument then shows that~\eqref{eq:quadratic_decay} implies $K^q_R(x,y) \geq K$ for all $x,y$ such that $d(x,y) \leq R$.

\subsubsection{Exponential curvature}

A classical lower Ollivier curvature bound $K$ is characterized via
\begin{align}\label{eq:OllivierGradient}
\|\nabla_R P_t f\|_\infty \leq e^{-Kt}\|\nabla_R f\|_\infty
\end{align}
The idea of Exponential Ollivier curvature is to strengthen the gradient estimate to
\begin{align}\label{eq:exp_decay}
\|\nabla_R \log P_t f\|_\infty \leq e^{-Kt}\|\nabla_R \log f\|_\infty
\end{align}
for positive $f$. This is indeed stronger than (\ref{eq:OllivierGradient}) since for bounded $f$,
\begin{align*}
\|\nabla_R P_t f\|_\infty = \lim_{\eps \to 0} \frac 1 \eps \|\nabla_R \log (1 + \eps P_t f)\|_\infty 
&= 
\lim_{\eps \to 0} \frac 1 \eps \|\nabla_R \log P_t(1 + \eps f)\|_\infty
\\&\leq  e^{-Kt} \lim_{\eps \to 0} \frac 1 \eps \|\nabla_R \log (1 + \eps f)\|_\infty \\&= e^{-Kt}\|\nabla_R f\|_\infty
\end{align*}
where we used stochastic completeness to guarantee $P_t 1 = 1$.

Let $x,y$ s.t. $d(x,y)\leq R$ and $\log f(y) - \log f(x) = \|\nabla_R \log f\|$. A calculation similar in spirit to the one in Section~\ref{sec:first_estimate} yields that the following condition is a consequence of~\eqref{eq:exp_decay}.
\[
\frac{\Delta f(x)}{f(x)} - \frac{\Delta f(y)}{f(y)} \geq K \|\nabla_R \log f\|.
\]
Taking $f = \exp(g)$, and minimizing yields the following definition.

\begin{definition}[Exponential Gradient-Ollivier Curvature]\label{def:exp_local}
\[
K_R^e(x,y) = \inf_{r>0}\frac 1 r \cdot \inf_{\substack{g: V\to \R \\ |\nabla g\|=r \\ g(y)-g(x)=r}} \frac{\Delta \exp(g)}{\exp(g)}(x) -\frac{\Delta \exp(g)}{\exp(g)}(y) 
\]
\end{definition}
Remark that $K_R^e$ is not necessarily symmetric.
By the considerations above, the gradient estimate (\ref{eq:exp_decay}) implies $K_R^e(x,y) \geq K$.

\subsection{Gradient estimates}\label{sec:estimates}

It turns out that not only do the gradient estimates~\eqref{eq:linear_decay}, \eqref{eq:quadratic_decay}, \eqref{eq:exp_decay} imply that the respective curvatures are bounded from below, but also vice versa. In fact, the obtained lower bounds are equivalent to the corresponding gradient estimates.

\begin{theorem}\label{thm:linear_equivalent}
Let $K \in \R$. The following are equivalent:
\begin{enumerate}[(i)]
\item For all $x,y$ such that $d(x,y) \leq R$ we have $K_R(x,y) \geq K$.
\item For all positive $f \in \ell^\infty(V)$ 
\[ |\nabla_R P_t f| \leq e^{-Kt} P_t |\nabla_R f|.\]
\end{enumerate}
\end{theorem}

\begin{theorem}\label{thm:quadratic_equivalent}
Let $K \in \R$. The following are equivalent:
\begin{enumerate}[(i)]
\item For all $x,y$ such that $d(x,y) \leq R$ we have $K^q_R(x,y) \geq K$.
\item For all $0 < f \in l^\infty(V)$ \[ \abs*{\nabla_R \sqrt{ P_t f}}^2 \leq e^{-2K t}  P_t \abs*{ \nabla_R \sqrt{f} }^2.\]
\end{enumerate}

\end{theorem}

\begin{theorem}\label{thm:exp_equivalent}
Let $K \in \R$. If $G$ is finite, then the following are equivalent:

\begin{enumerate}[(i)]
\item For all $x,y$ such that $d(x,y) \leq R$ we have $K^e_R(x,y) \geq K$.
\item For all positive $0 < f \in \ell^\infty(V)$ 
\[ \|\nabla_R \log P_t f\|_\infty \leq e^{-Kt}\|\nabla_R \log f\|_\infty.\] 
\end{enumerate}

\end{theorem}

In the case of non-negative gradient- (or quadratic gradient-) curvature we can still get a decaying estimate for $\abs{\nabla_R P_t f}$ (or $\abs{\nabla_R \sqrt{P_t f}}$ respectively), as long as the graph has a uniform degree bound. Let $\D$ denote the largest degree in $G$. 

\begin{theorem}\label{thm:nonegative_decay} Let $R \in \N$. 
\begin{enumerate}[(i)]
\item Suppose $K_R(x,y) \geq 0$ for all $x,y \in V$ such that $d(x,y)\leq R$ Then 
\[
|\nabla_R P_t f|^2 
\leq
\frac{\|f\|_\infty^2}{t} \cdot 2R \D^{R-1}
\]
\item Suppose $K^q_R(x,y) \geq 0$ for all $x,y \in V$ such that $d(x,y)\leq R$. Then
 \[
|\nabla_R \sqrt{P_t} f|^2
\leq
\frac{\norm{f \log f}_\infty}{t} \cdot 2R \Deg_{\max}^{R-1}
\]
\end{enumerate}
\end{theorem}

Using Theorem~\ref{thm:linear_decay} (i), one can establish Buser's inequality following the arguments of \cite{liu2014eigenvalue} and the Harnack inequality
\[
|\nabla_R f|^2 \leq 2eR \D^{R-1} \lambda \|f\|_\infty^2
\]
for eigenfunctions $f$ to the eigenvalue $\lambda$ of $-\Delta$ by using $P_t f = e^{-\lambda t} f$ and choosing $t=1/\lambda$.
\subsection{Curvature bounds}\label{sec:nonnegative}

In this section we describe the combinatorial characterization of non-negative curvature for each of the variants. Somewhat surprisingly these turn out to be the same, even though the numerical values of the curvatures may be different.

The following combinatorial definition plays a crucial role in characterizing non-negative curvature.

\begin{definition}[$R$-transport map]
We call a function $\phi: B_1(y)\setminus B_R(x) \to B_1(x)$ a $y\to x$ $R$-transport map, if it's injective and satisfies $d(z,\phi(z))\leq R$ for all $z \in B_1(y)\setminus B_R(x)$.

More generally, we allow transport maps to be defined on a superset of $B_1(y) \setminus B_R(x)$: Suppose there is a set $A$ such that $B_1(y)\setminus B_R(x) \subset A \subset B_1(y)$. Then a function $\phi: A \to B_1(x)$ is a $y \to x$ $R$-transport map if it is injective and satisfies $d(z,\phi(z)) \leq R$ for all $z \in A$. 

We define the \emph{defect} of an $R$-transport map $\phi : A \to B_1(x)$ as $\dfct(\phi) = \abs{B_1(x) \setminus (\phi(A) \cup B_R(y))}$. Finally we let $\dfct(y \to x) = \inf\{ \dfct(\phi) | \phi \mbox{ is an $y\to x$ $R$-transport map}\}$. Clearly, $\dfct(y\to x) \geq 0$. If there are no $y\to x$ $R$-transport maps, we interpret the defect as $\infty$. 
\end{definition}

\begin{remark}\label{rem:equivalent}
If $\phi$ is a $y\to x$ $R$-transport map such that $\dfct(\phi)=0$, then $\phi^{-1}$ is an $x\to y$ $R$-transport map, and $\dfct(\phi^{-1})=0$ as well. Hence $\dfct(y\to x) =0$ implies that $\dfct(x\to y)=0$.
\end{remark}

Using the terminology of transport maps, our main results regarding non-negative curvature are the following.

\begin{theorem}\label{thm:combinatorial}~
\begin{enumerate}[(i)]
\item \label{item:infinity} If there is no $y\to x$ $R$-transport map, then $K_R(x,y) = K_R^q(x,y) = K^e_R(x,y) = -\infty$.
\item \label{item:lowerbound} If there is a $y\to x$ $R$-transport map, then $K_R(x,y) \geq -\dfct(y \to x), K_R^e(x,y) \geq -\dfct(y\to x)$, and $K^q_R(x,y) \geq -3/2 \dfct(y\to x)$.
\item \label{item:equivalence}The following are equivalent
\begin{enumerate}
\item Both $y\to x$ and $x \to y$ $R$-transport maps exist.
\item $\dfct(y \to x) = 0$.
\item $K_R(x,y) \geq 0$ and $K_R(y,x) \geq 0$.
\item $K_R^q(x,y) \geq 0$ and $K_R^q(y,x) \geq 0$.
\item $K_R^e(x,y) \geq 0$ and $K_R^e(y,x) \geq 0$.
\end{enumerate} 
\end{enumerate}
\end{theorem}

\begin{corollary} The hexagonal lattice has curvature $K_2 \geq 0$. First, it's sufficient to check $K_2(x,y) \geq 0$ for $x,y$ exactly distance 2 apart, otherwise $B_1(y) \setminus B_2(x)$ is empty. If $d(x,y)=2$ then ``translation by $\overrightarrow{yx}$'' defines a 2-transport map from $y$ to $x$. 
\end{corollary}

\section{Proofs}\label{sec:proofs}

\subsection{Gradient-Ollivier curvature}

\begin{proof}[Proof of Theorem~\ref{thm:linear_equivalent}] We have seen in Section~\ref{sec:first_estimate} that $(ii) \Rightarrow (i)$.
The converse implication 
follows from Lemma~\ref{lem:Gincreasing} below, by comparing $G(0)$ and $G(t)$. 
\end{proof}

\begin{lemma}\label{lem:Gincreasing}
Suppose $K_R(x,y) \geq K$. Let $t >0$ fixed.
Let us define the function $G_s: V \to \R$ via $G_s := e^{-Ks} P_s|\nabla_R P_{t-s} f|$.  Then $G_s$ is pointwise increasing for $s \in [0,t]$.
\end{lemma}
\begin{proof}
Let us denote $\underline \partial_s G_s := \liminf_{h \to 0} \frac{G_{s+h}-G_s}h$.
We calculate 
\begin{align*}
\underline \partial_s G_s = -K\cdot G_s + e^{-Ks}\left( P_s ( \underline \partial_s|\nabla_R P_{t-s} f|) + 
\Delta P_s |\nabla_R P_{t-s} f| \right).
\end{align*}
By \eqref{eq:partial_gradient}, for any $y \in N_R(x)$ such that $\abs{\nabla_R P_{t-s}f}(x) = \abs{P_{t-s}f(y)-P_{t-s}f(x)}$, we have
\begin{align*}
\underline \partial_s|\nabla_R P_{t-s} f|(x) & \geq \partial_s \abs{P_{t-s}f(y) - P_{t-s}f(x)} \\&= (\partial_s P_{t-s}f(y) - \partial_s P_{t-s}f(x))\sgn(P_{t-s}f(y)-P_{t-s}f(x)) \\&= -(\Delta P_{t-s}f(y) - \Delta P_{t-s}f(x))\sgn(P_{t-s}f(y)-P_{t-s}f(x))
\\&\geq  |\nabla_R P_{t-s} f|(x) K_R(x,y) - \Delta |\nabla_R P_{t-s} f|(x)
\\&\geq K |\nabla_R P_{t-s} f|(x) -  \Delta |\nabla_R P_{t-s} f|(x).
\end{align*}
Combining with the above equation and using $P_s \Delta |\nabla_R P_{t-s} f|= \Delta P_s |\nabla_R P_{t-s} f|$ yields $\underline \partial_s G(s) \geq 0$ proving the claim.
\end{proof}

We now want to prove that under non-negative curvature,
\[
|\nabla_R P_t f|^2 \leq \frac C t \|f\|_\infty^2.
\]
To do so, we will upper bound $|\nabla_R f|^2$ by $\Gamma f$. However, this cannot work pointwise if $R \geq 2$ since $\Gamma$ only depends the one-step neighborhood. The idea to overcome this issue is to sum up $\Gamma f$ over a large enough neighborhood.

\begin{definition}[The averaging operator $A$]
Let $G=(V,w,m)$ be a graph. We write $Q(x,y):=w(x,y)/m(x)$.
We have 
\[
\Delta f(x) = \frac 1 {m(x)} \sum_y w(x,y) (f(y)-f(x)) = 
\sum_y Q(x,y) (f(y)-f(x)).
\]
We write $\Deg(x):=\sum_y Q(x,y)$ and $\Qmin := \inf_{x\sim y} Q(x,y)$.
We define $A:= \Delta + \D$.
\end{definition}
The operator $A$ can be seen as an averaging operator over the ball of radius one.
Note that $Af \geq 0$ whenever $f \geq 0$ and $\|A f\|_\infty \leq \D \|f\|_\infty$. Moreover, $P_t \circ A = A \circ P_t$. Now we are ready to bound $\abs{ \nabla_R f}^2$ in terms of $\Gamma f$ on the $R$-neighborhood.

\begin{lemma}\label{lem:GammaNabla} For all functions $f:V\to \R$, one has
\[
|\nabla_R f|^2 \leq 2R/\Qmin^R A^{R-1} \Gamma f.
\] 
\end{lemma}

\begin{proof}
Let $x=x_0 \in V$.
We first estimate $A^{R-1} \Gamma f (x)$.
For all $g \geq 0$ observe
$Ag(x_0) \geq \sum_{x_1 \sim x_0} \Qmin g(x_1)$
Therefore,
\[
A^{R-1} \Gamma f (x_0) \geq \Qmin^{R-1}\sum_{x_0 \sim \ldots \sim x_{R-1}} \Gamma f(x_{R-1}).
\]
Similarly, $\Gamma f(x) \geq \frac 1 2\Qmin \sum_{z \sim x} (f(z)-f(x))^2$. Therefore,
\begin{align}\label{eq:AGammaSumSquares}
A^{R-1} \Gamma f (x_0) \geq \frac 1 2\Qmin^{R}\sum_{x_0 \sim \ldots \sim x_{R}}  \left(f(x_{R-1}) - f(x_R)\right)^2.
\end{align}
Now, we compute $|\nabla_R f|(x)$.
Let $y \in B_R(x)$ s.t. $|f(x)-f(y)|= |\nabla_R f|(x)$. Let $x=y_0\sim \ldots \sim y_r = y$ be a path from $x$ to $y$ with $r\leq R$.
Picking the corresponding elements from the sum in (\ref{eq:AGammaSumSquares}) yields
\begin{align*}
A^{R-1} \Gamma f (x_0) \geq \frac 1 2\Qmin^{R} \sum_{k=1}^r (f(y_k)-f(y_{k-1}))^2 \geq \frac 1 2\Qmin^{R} \frac 1 r(f(y_r)-f(y_0))^2 \geq \frac 1 2\Qmin^{R} \frac 1 R |\nabla_R f|^2.
\end{align*}
Rearranging proves the lemma.
\end{proof}

We write $\dim(G):=\D/\Qmin$. The following is a generalization of Theorem~\ref{thm:nonegative_decay} (i).
\begin{theorem}\label{thm:linear_decay}
Let $R \in \N$ and $K \in \R$ fixed. Suppose $K_R(x,y) \geq K$ for all $x,y \in V$ such that $d(x,y) \leq R$.  Then,
\[
\frac{e^{2Kt} -1}{2K}|\nabla_R P_t f|^2 
\leq
2R \frac{\dim(G)^{R-1}}{\Qmin}\|f\|_\infty^2 
\]
where we set $\frac{e^{2Kt} -1}{2K}:=t$ if $K=0$.
\end{theorem}

\begin{proof}
Let
\[
H_s:= A^{R-1}P_s\left[(P_{t-s} f)^2\right]
\]
and, as in Lemma~\ref{lem:Gincreasing}, let
\[
G_s:= e^{-Ks} P_s |\nabla_R P_{t-s} f|.
\]
Taking derivative of $H$ and employing Lemma~\ref{lem:GammaNabla} yields
\begin{align*}
\partial_s H_s
= 
P_s A^{R-1}\Gamma P_{t-s} f
&\geq 
\frac{\Qmin^R}{2R} P_s \left(|\nabla_R P_{t-s} f|^2 \right)
\\&\geq 
\frac{\Qmin^R}{2R} \left(P_s |\nabla_R P_{t-s} f|\right)^2
= 
\frac{\Qmin^R}{2R} 
e^{2Ks}G_s^2.
\end{align*}
Integrating and using that $G_s$ is increasing in $s$ on $[0,t]$ due to Lemma~\ref{lem:Gincreasing} yields
\begin{align*}
H_t-H_0 \geq \int_0^t \frac{\Qmin^R}{2R} 
e^{2Ks}G_s^2 ds \geq \frac{\Qmin^R}{2R}G_0^2 \int_0^t e^{2Ks}ds = \frac{\Qmin^R}{2R}G_0^2 \frac{e^{2Kt} -1}{2K}
\end{align*}
where the latter term $\frac{e^{2Kt} -1}{2K}$ is replaced by $t$ in the case $K=0$.
Plugging in $H_t$ and $G_0$, and using $\|Ag\|_\infty \leq \D \|g\|_\infty$ and $H_0 \geq 0$ yields
\[
\D^{R-1} \|f\|_\infty^2 \geq 
A^{R-1}(P_t f)^2 = H_t \geq \frac{\Qmin^R}{2R}G_0^2 \frac{e^{2Kt} -1}{2K} =  \frac{\Qmin^R}{2R}|\nabla_R P_t f|^2 \frac{e^{2Kt} -1}{2K}.
\]
Rearranging gives us
\[
\frac{e^{2Kt} -1}{2K}|\nabla_R P_t f|^2 
\leq
2R \frac{\dim(G)^{R-1}}{\Qmin}\|f\|_\infty^2 
\]
which proves the theorem.
\end{proof}

\subsection{Quadratic curvature}

\begin{proof}[Proof of Theorem~\ref{thm:quadratic_equivalent}]
Fix $x \in V(G)$. For any $f > 0$ we can compute, using the $ g= \sqrt{f}$ notation, that
\begin{equation}\label{eq:partial1} \partial_t \left(P_t \abs*{\nabla_R \sqrt{f}}^2 \right)_{t=0} = \Delta \abs*{\nabla_R \sqrt{f}}^2 = \Delta \abs{\nabla_R g}^2 \end{equation} and
\begin{multline}\label{eq:partial2} \partial_t 
\left(\abs*{\nabla_R \sqrt{P_t f}}^2 \right)_{t=0} 
= \sup_{\substack{y: d(x,y) \leq R \\ \abs{\sqrt{f(y)}-\sqrt{f(x)}}  = \abs{\nabla_R \sqrt{f}}(x)}}
\partial_t\left(\sqrt{P_t f}(y) - \sqrt{P_t f}(y) \right)^2_{t=0}   =\\
= \sup_{\substack{y: d(x,y) \leq R \\ g(y)-g(x) = \abs{\nabla_R g}(x)}}(g(y)-g(x))\left(\frac{\Delta g^2}{g}(y)- \frac{\Delta g^2}{g}(x)\right)
\end{multline}
Having this calculation, we first prove $(i)\Rightarrow (ii)$.
By $(i)$, we have 
\[
\partial_t \left( \abs*{\nabla_R \sqrt{P_t f}}^2 \right)_{t=0} \leq \partial_t \left(  e^{-2Kt}P_t \abs*{\nabla_R \sqrt{f}}^2 \right)_{t=0}
= -2K \abs{\nabla_R g}^2 + \partial_t \left( P_t \abs*{\nabla_R \sqrt{f}}^2 \right)_{t=0},
\]
and plugging in \eqref{eq:partial1} and \eqref{eq:partial2} we get 
\[ 2K \abs{\nabla_R g}^2 \leq \Delta \abs{\nabla_R g}^2  -\sup_{\substack{y: d(x,y) \leq R \\ g(y)-g(x) = \abs{\nabla_R g}(x)}}(g(y)-g(x))\left(\frac{\Delta g^2}{g}(y)- \frac{\Delta g^2}{g}(x)\right) .\]
Scaling so that $\abs{\nabla_R g}(x) = 1$, and taking infimum in $g$, we get 
\[ 2K \leq \inf_{|\nabla_R g|(x)=1} \inf_{|g(y)-g(x)|=1}\Delta |\nabla_R g|^2(x) - (g(y)-g(x)) \cdot \left(\frac{\Delta g^2}{g}(y)- \frac{\Delta g^2}{g}(x)\right) = 2 \inf_{y}K^q_R(x,y).\]
Thus $K \leq K^q_R(x,y)$ for all edges $(xy)$.

For the reverse direction let us fix $t > 0$ and define 
\begin{equation}\label{eq:quadratic_Gs} G_s = P_s \abs*{\nabla_R \sqrt{P_{t-s} f}}^2 e^{-2Ks}\end{equation}
and compute, using the $g = \sqrt{P_{t-s} f}$ abbreviation, that 
\begin{multline*} \partial_s G_s  = -2K G_s + e^{-2Ks} \Delta P_s \abs*{\nabla_R \sqrt{P_{t-s} f}}^2  + e^{-2Ks} P_s \left( \partial_s  \abs*{\nabla_R \sqrt{P_{t-s} f}}^2 \right) = \\ = e^{-2Ks} P_s \left( -2K \abs*{\nabla_R g}^2 + \Delta \abs*{\nabla_R g}^2  -\sup_{\substack{y: d(x,y) \leq R \\ g(y)-g(x) = \abs{\nabla_R g}(x)}}(g(y)-g(x))\left(\frac{\Delta g^2}{g}(y)- \frac{\Delta g^2}{g}(x)\right)\right)  \geq 0. \end{multline*}
Here we used \eqref{eq:partial2} to compute the derivative inside $P_s$. The last inequality follows from the curvature assumption $K^q_R(x,y) \geq K$. 

Since $\partial_s G_s \geq 0$ for all $0 \leq s \leq t$, we get that $G_0 \leq G_t$ which is what we wanted to prove.
\end{proof}

Now we set out to prove a generalization of Theorem~\ref{thm:nonegative_decay} (ii). 
\begin{theorem}\label{thm:quadratic_decay}
Let $R \in \N$ and $K \in \R$ fixed. Suppose $K_R^q(x,y) \geq K$ for all $x,y \in V$ such that $d(x,y) \leq R$.  Then,
\[
\frac{e^{2Kt} -1}{2K}|\nabla_R \sqrt{P_t} f|^2
\leq
2R \frac{\dim(G)^{R-1}}{\Qmin}\norm{f \log f}_\infty
\]
where we set $\frac{e^{2Kt} -1}{2K}:=t$ if $K=0$.

\end{theorem}

\begin{proof} We proceed analogously to Theorem~\ref{thm:linear_decay}. Let us introduce the following function:
\begin{equation}\label{eq:quadratic_Hs}
H_s = A^{R-1} P_s( P_{t-s} f \log(P_{t-s} f)).
\end{equation}
The idea is to relate $\partial_s H_s$ to the $G_s$ defined in \eqref{eq:quadratic_Gs}. To do so, we need the following elementary inequality.

\begin{lemma}\label{lem:logarithmic_inequality}
Let $a,b > 0$. Then
\[ a \log a - b \log b - (a-b)\log b - (a-b) \geq (\sqrt{a} - \sqrt{b})^2 \]
\end{lemma}
\begin{proof}
After simple algebraic manipulation and letting $c = \sqrt{b/a}$, this reduces to the standard $\log c \leq c-1$ inequality.
\end{proof}

Now we compute, using the $g= P_{t-s}f$ notation,
\[ \partial_s H_s = A^{R-1} P_s \left(\Delta(g \log g) -(\log g) \Delta g - \Delta g\right).\]
At a node $x$ the expression $\Delta(g \log g) -(\log g) \Delta g - \Delta g$ can be bounded from below using Lemma~\ref{lem:logarithmic_inequality} as  
\begin{multline*} (\Delta(g \log g) -(\log g) \Delta g - \Delta g)(x) =  \sum_{y \sim x} g(y) \log g(y) - g(x) \log g(x) -(g(y)-g(x))\log g(x) - (g(y)-g(x)) \geq \\ \geq \sum_{y \sim x} (\sqrt{g(y)}- \sqrt{g(x)})^2 = \Gamma(\sqrt{g})(x).\end{multline*}
Hence, according to Lemma~\ref{lem:GammaNabla}, and recalling the definition of $G_s$ from \eqref{eq:quadratic_Gs}, we find that
\[ \partial_s H_s \geq A^{R-1} P_s(\Gamma(\sqrt{g})) = P_s\left(A^{R-1} \Gamma(\sqrt{P_{t-s} f})\right)\geq  \frac{Q^R_{\min}}{2R} P_s\left( \abs{\nabla_R \sqrt{P_{t-s} f}}^2\right) = \frac{Q^R_{\min}}{2R} e^{2Ks} G_s
\]
Integrating and using that $G_s$ is increasing in $s$ on $[0,t]$  as in the proof of Theorem~\ref{thm:quadratic_equivalent}
\begin{align*}
H_t-H_0 \geq \int_0^t \frac{\Qmin^R}{2R} 
e^{2Ks}G_s ds \geq \frac{\Qmin^R}{2R}G_0 \int_0^t e^{2Ks}ds = \frac{\Qmin^R}{2R}G_0\frac{e^{2Kt} -1}{2K}
\end{align*}
where the latter term $\frac{e^{2Kt} -1}{2K}$ is replaced by $t$ in the case $K=0$.
Plugging in $H_t$ and $G_0$, and using $\|Ag\|_\infty \leq \D \|g\|_\infty$ and $H_0 \geq 0$ yields
\[
\D^{R-1} \norm{f \log f}_\infty \geq 
A^{R-1}P_t (f \log f) = H_t \geq \frac{\Qmin^R}{2R}G_0 \frac{e^{2Kt} -1}{2K} =  \frac{\Qmin^R}{2R}|\nabla_R \sqrt{P_t} f|^2 \frac{e^{2Kt} -1}{2K}.
\]
Rearranging gives us
\[
\frac{e^{2Kt} -1}{2K}|\nabla_R \sqrt{P_t} f|^2
\leq
2R \frac{\dim(G)^{R-1}}{\Qmin}\norm{f \log f}_\infty
\]
which proves the theorem.

\end{proof}
\subsection{Exponential curvature}

\begin{proof}[Proof of Theorem~\ref{thm:exp_equivalent}]

As (ii) $\Rightarrow$ (i) follows in a similar (but even simpler) way than the same part of Theorem~\ref{thm:linear_equivalent}, we only need to prove (i) $\Rightarrow$ (ii). Here we assume $G$ is finite, so $\norm{\nabla_R \log P_t f}_{\infty}$ is attained somewhere. We compute, using $g_0 = \log P_t f$ and $r_0 = \norm{\nabla_R g_0}_\infty$, that 
\begin{multline*} \partial_t \norm{\nabla_R \log P_t f}_\infty  = \sup_{\substack{x, y : d(x,y) \leq R \\ \log P_t f(y) - \log P_t f(x) = \norm{\nabla_R \log P_t f}_{\infty}}} \partial_t \left( \log P_t f(y) - \log P_t f(x)\right) = \\ =  -\inf_{\substack{x, y : d(x,y) \leq R \\ \log P_t f(y) - \log P_t f(x) = \norm{\nabla_R \log P_t f}_\infty}}  \frac{\Delta P_t f}{P_t f}(x) - \frac{\Delta P_t f}{P_t f}(y) = \\ = -r_0 \inf_{\substack{x,y: d(x,y)\leq R\\g_0(y)-g_0(x) = \norm{\nabla_R g_0}_\infty = r_0    }}\frac{1}{r_0} \left( \frac{\Delta \exp(g_0)}{\exp(g_0)}(x) -\frac{\Delta \exp(g_0)}{\exp(g_0)}(y)\right)  \leq \\ -r_0 \inf_{\substack{x,y, g, r: d(x,y)\leq R \\  g(y)-g(x)=\norm{\nabla_R g}_\infty = r }} \frac{1}{r} \left( \frac{\Delta \exp(g)}{\exp(g)}(x) -\frac{\Delta \exp(g)}{\exp(g)}(y)\right) = \\=  -r_0 \inf_{x,y: d(x,y)\leq R} K_R^e(x,y) \leq -r_0 K = -K \norm{\nabla_R \log P_t f}_\infty \end{multline*}
From this (ii) follows by integration. 
\end{proof}

\subsection{Transport maps and curvature lower bounds}

We now prove the combinatorial characterizations of non-negative curvature. 



\begin{proof}[Proof of Theorem~\ref{thm:combinatorial} (\ref{item:infinity})]
From the assumption it follows by Hall's theorem that there is a subset $A \subset B_1(y) \setminus B_R(x)$ such that $\abs{N_R(A)} < \abs{A}$ where 
\[ N_R(A) = \{ z \in B_1(x) : \exists a\in A , d(a,z) \leq R \} . \] We will use these subsets to construct test-functions that exhibit that the various curvatures cannot be bounded from below.

To show $K_R^e(x,y) = -\infty$, let us fix $r$ and define 
\[ g(v) = \left\{ \begin{array}{ll} 2r &: v \in A \\ r &: 1\leq d(v,A) \leq R \\ 0 &: d(v,A) >R\end{array}\right. .\] Note that $A$ is disjoint from $B_R(x)$ by definition, thus $d(x,A) = R+1$ so $g(x)=0$.  Clearly $d(y,A)=1$ so $g(y)=r$. Thus $\norm{\nabla_R g}_\infty = r = g(y)-g(x)$. By Definition~\ref{def:exp_local} we have 
\[ K_R^e(x,y) \leq \frac{1}{r} \left( \frac{\Delta e^g}{e^g}(x) -\frac{\Delta e^g}{e^g}(y)  \right)  = \frac{1}{r} \left( (\abs{N_R(A)}-\abs{A)})(e^r-1) + \abs{B_1(y) \setminus B_R(A)}(e^{-r}-1)\right).\] The right hand side goes to $-\infty$ as $r \to \infty$ since $|A| > |N_R(A)|$, so $K_R^e(x,y) = -\infty$.

Next we show $K_R^q(x,y) = -\infty$. Let $\ep > 0 $ be fixed such that $|A| > (1+\ep)|N_R(A)|$. For some fixed $r$ let
\[ g(v) = \left\{ \begin{array}{ll} r &: v \in A \\ \ep &: v = x \\ 1+\ep &: v \not \in A\cup \{x\} \end{array}\right. .\]
We have for $v \in B_1(x)$ that $\abs{\nabla_R g}(v) = 1$ if $v \not \in N_R(A)$ and $\abs{\nabla_R g}(v) = r-1-\ep$ if $v\in N_R(A)$. So
\[ \Delta \abs*{\nabla_R g}^2(x) = |N_R(A)|(r^2+O(r)). \] Further we have $\Delta g^2 (x) = O(1)$ and \[\frac{\Delta g^2}{g}(y) = \frac{|A|}{1+\ep} r^2 + O(1).\] Then by Definition~\ref{def:quadratic_local} we get 
\[2 K_R^q(x,y) \leq   \Delta \abs*{\nabla_R g}^2(x) - \left(\frac{\Delta g^2}{g}(y) - \frac{\Delta g^2}{g}(x)\right) = r^2 \left(\abs{N_R(A)} - \frac{|A|}{1+\ep}\right) + O(r),\] and this clearly goes to $-\infty$ as $r \to \infty$ since $|A| > (1+\ep)|N_R(A)|$. This contradicts (ii). 

Using the same family of $g$ functions, a similar but shorter computation also gives $K_R(x,y) = -\infty$.
\end{proof}

\begin{proof}[Proof of Theorem~\ref{thm:combinatorial} (\ref{item:lowerbound})]
Let us fix a $y\to x$ $R$-transport map $\phi: A \to B_1(x)$ for which $\dfct(\phi) = \dfct(y\to x)$. Here $B_1(y) \setminus B_R(x) \subset A \subset B_1(y)$. Denote $C = B_1(x) \setminus (\phi(A) \cup B_R(y))$, so $\dfct(y\to x) = \dfct(\phi) = |C|$.
 Let $A' =\phi(A)$. Let further $H = B_1(y) \setminus A$ and $H' = (B_1(x) \setminus A') \cap B_R(y)$.  

We start by showing $K_R^e(x,y) \geq - \dfct(y \to x)$, which is the simplest. Take $g$ such that $\norm{\nabla_R g}_\infty = r$ and that $g(y) - g(x) = r$. For $z \in H$ we have $g(z) \leq g(x) + r =  g(y)$. We have the following simple consequences of the setup:
\[ \begin{array}{lll} 
z\in H  &\Rightarrow g(z) \leq g(y)  &\Rightarrow \exp(g(z)-g(y)) -1 \leq 0\\
z\in H'  &\Rightarrow g(z) \geq g(y)-r = g(x) &\Rightarrow \exp(g(z)-g(x)) -1 \geq 0\\
z\in C &\Rightarrow g(z) \geq g(x)-r &\Rightarrow \exp(g(z)-g(x) -1 \geq e^{-r}-1 \geq -r\\
z\in A &\Rightarrow g(\phi(z)) \geq g(z) -r &\Rightarrow \exp( g(\phi(z))-g(x)) - \exp(g(z) - g(y)) \geq 0
\end{array}\]
Then we can estimate
\begin{multline*} \frac{\Delta e^{g}}{e^{g}}(x) - \frac{\Delta e^{g}}{e^{g}}(y) = \sum_{z\in A} e^{ g(\phi(z))-g(x)} - e^{g(z) - g(y)} + \sum_{z\in H' \cup C}\left( e^{g(z)-g(x)}-1\right) -\sum_{z\in H} \left(e^{g(z)-g(y)}-1\right) \geq \\ \geq \sum_{z\in A}  e^{ g(z)-g(x)-r} - e^{g(z) - g(y)} + \sum_{z\in H'} 0 + \sum_{z\in C} -r   - \sum_{z\in H} 0 \geq -r |C|, \end{multline*}
Thus
\[ K_R^e(x,y) = \inf_{r>0}\frac 1 r \cdot \inf_{\substack{g: \|\nabla g\|=r \\ g(y)-g(x)=r}} \frac{\Delta \exp(g)}{\exp(g)}(x) -\frac{\Delta \exp(g)}{\exp(g)}(y) \geq -|C| = -\dfct(y\to x) \] as claimed.

Next we show $K_R(x,y) \geq -\dfct(y\to x)$. Let $f$ be a possible test function. After possibly muliplying $f$ by $-1$ we can assume that $f(y) = f(x)+1$.
 We have the following simple consequences of the setup:
\[ \begin{array}{ll} 
z\in H  &\Rightarrow f(z) \leq f(y) \\
z\in H'  &\Rightarrow \abs{\nabla_R f}(z) \geq f(y) -f(z) = f(x)+1 - f(z)\\
z\in C &\Rightarrow \abs{\nabla_Rf}(z) \geq |f(z)-f(x)| \\
z\in A &\Rightarrow \abs{\nabla_R f}(z) \geq f(z)-f(\phi(z))
\end{array}\]
Then we can estimate
\begin{multline*} \Delta \abs{\nabla_R f}(x) - \Delta f(y) + \Delta f(x)  = \sum_{z \in A} \left(\abs{\nabla_R f}(\phi(z)) - 1 -(f(z)-f(y))+ f(\phi(z))-f(x) \right) -  \\ - \sum_{z\in H}\left( f(z) -f(y)\right) +\sum_{z\in H' \cup C} \left(\abs{\nabla_R f}(z) - 1 + f(z)-f(x) \right) \geq \\ \geq  \sum_{z\in A} \left(f(z) - f(\phi(z)) - 1 -f(z) +f(y) +f(\phi(z)) - f(x)\right) + \sum_{z\in C} \left( \abs{f(z)-f(x)} +f(z)-f(x) -1\right) \geq \\ \geq -|C|.
\end{multline*} 
Thus
\[
K_R(x,y) := \inf_{f(y)-f(x)=|\nabla_R f|(x)=1} \Delta \abs{\nabla_R f}(x) - \Delta f(y) + \Delta f(x) \geq  -|C|= -\dfct(y\to x) 
\] as claimed.

Finally we turn to showing $K_R^q(x,y) \geq -3/2 \dfct(y\to x)$. Let $g > 0$ be a possible test function.
First assume $g(y) = 1+ g(x) = 1+ c$. 
 We have the following simple consequences of the setup:
\[ \begin{array}{ll} 
z\in H  &\Rightarrow g(z) \leq g(y) \\
z\in H'  &\Rightarrow \abs{\nabla_R g}(z) \geq |g(y) -g(z)| = |c+1 - g(z)|\\
z\in C &\Rightarrow \abs{\nabla_R g}(z) \geq |g(z)-g(x)| \\
z\in A &\Rightarrow \abs{\nabla_R g}(z) \geq |g(z)-g(\phi(z))|
\end{array}\]
Then we can estimate
\begin{multline*} \Delta \abs*{\nabla_R g}^2(x) - \frac{\Delta g^2}{g}(y) + \frac{\Delta g^2}{g}(x) =\\= \sum_{z \in A} \left( \abs{\nabla_R g(\phi(z))}^2 - 1 - \frac{g^2(z) - (1+c)^2}{1+c} + \frac{g^2(\phi(z)) - c^2}{c}  \right) - \\ -\sum_{z \in H} \frac{g(z)^2 - g(y)^2}{g(y)} + \sum_{z \in H' \cup C} \left(\abs{\nabla_R g}^2(z) -1 + \frac{g(z)^2-c^2}{c}    \right) \geq \\ \geq \sum_{z \in A} \left( (g(z)- g(\phi(z))^2 - \frac{g(z)^2}{1+c} + \frac{g(\phi(z))^2}{c}  \right) + \sum_{z\in C} \left( (g(z)-c)^2 -1 -c +\frac{g(z)^2}{c} \right) + \\ + \sum_{z\in H'} \left( (g(z)-c-1)^2 -1 -c +\frac{g(z)^2}{c} \right) 
= \\ =
 \sum_{z\in A} \left( g(z) \sqrt{\frac{c}{1+c}} - g(\phi(z)) \sqrt{\frac{1+c}{c}} \right)^2 + \sum_{z\in C} \left( (g(z)-c+1)^2 -2 +\frac{(g(z)-c)^2}{c} \right)  
 +\\+
 \sum_{z \in H'} \left((g(z)-c)^2 + \left(\sqrt{c} - \frac{g(z)}{\sqrt{c}} \right)^2 \right)
\geq
 -2|C|.
\end{multline*}

Next we consider the case $g(x) = 1+g(y) = 1+c$. Everything works the same way, except the sum over $z\in H'$ becomes, using the notation $\alpha(z) = g(z) - g(x) = g(z) - 1 -c $, and observing that $|\alpha| \leq 1$ by the condition on $|\nabla_R g|(x)$, 
\begin{multline*}
\sum_{z \in H'\cup C} \left( \abs{\nabla_R g}^2(z) - 1 - \frac{g(z)^2 - (c+1)^2}{c+1} \right) 
\geq\\ \geq
\sum_{z \in C} \left( (g(z) - 1 -c )^2 - 1 - \frac{g(z)^2 - (c+1)^2}{c+1} \right) 
+ \sum_{z \in H'} \left( (g(z) - c )^2 - 1 - \frac{g(z)^2 - (c+1)^2}{c+1} \right) 
=\\=
\sum_{z \in C} \left( \alpha(z)^2 - 1 - 2\alpha(z) - \frac{\alpha(z)^2}{c+1}\right) + \sum_{z\in H'} c\left(\frac{g(z)}{\sqrt{c+1}} - \sqrt{c+1} \right)^2 \geq -3 |C|.
\end{multline*}
 
So from the two cases combined $K_R^q(x,y) \geq -\frac{3}{2}|C| = -\frac{3}{2}\dfct(y\to x)$.

This completes the proof.
\end{proof}

Finally we prove the equivalent characterizations of non-negative curvature.
\begin{proof}[Proof of Theorem~\ref{thm:combinatorial} (\ref{item:equivalence})]
By Remark~\ref{rem:equivalent} we know that (b) implies $\dfct(x\to y) =0$ as well. Thus according to part (\ref{item:lowerbound}) of this theorem, we see that (b) implies (c), (d), and (e). On the other hand, according to part (\ref{item:infinity}) of this theorem, we see that (c), (d), and (e) all imply (a). Hence it remains to show that (a) implies (b).

However, that (a) implies (b) follows from a standard fact in the theory of matchings in bipartite graphs. Construct an auxiliary bipartite graph whose two vertex sets are the nodes in $B_1(x)$ and $B_1(y)$ respectively. Edges correspond to pairs of nodes whose distance is at most $R$. Then a $y\to x$ $R$-transport map is simply a matching in this graph that covers $B_1(y)\setminus B_R(x)$, while an $x\to y$ $R$-transport map is a matching that covers $B_1(x) \setminus B_R(y)$. 

It is then a standard fact that there is a matching that covers both of these sets. (One takes the union of the two matchings and looks at the connected components, which can only be even cycles or paths. One can then select appropriate matchings in each type independently.) Such a matching can then be thought of as a $y\to x$ $R$-transport map $\phi$ and by definition $\dfct(\phi) = 0$.

\end{proof}

As a corollary, we can show that the classical Ollivier curvature can be bounded from below using $K_R$ for $R=1$.

\begin{proof}[Proof of Theorem~\ref{thm:Ollivier_comparison}]
We have
\[
K_1^{Ol}(x,y)= \min_{\substack{\|\nabla f\|_\infty=1\\f(y)-f(x)=1}} \Delta f(x)-\Delta f(y)
\]
Let $f$ be a minimzing function.
Then,
\[
 |\nabla_1 f|(x) = 1
\]
and for $z\sim x$,
\[
 |\nabla_1 f|(z) \leq 1
\]
giving
\[
\Delta |\nabla_1 f|(x) \leq 0
\]
and thus,
\[
K_1(x,y) \leq  \Delta |\nabla_1 f|(x) + K_1^{Ol} \leq K_1^{Ol}(x,y)
\]
which finishes the proof.
\end{proof}

\paragraph*{Acknowledgements:} 
Florentin Munch and Mark Kempton thank the Harvard CMSA, where the bulk of this research was carried out, for their hospitality. Munch also acknowledges the support of the Sudienstiftung des deutschen Volkes. Gabor Lippner acknowledges support of a Simons Foundation Collaboration Grant for Mathematicians.


\bibliographystyle{plain}
\bibliography{Bibliography}

\end{document}